\documentclass[a4paper, 10pt, reqno, english, ones]{amsart}  
\usepackage[T1]{fontenc}

\usepackage[
    a4paper,
    left=3.0cm,
    right=3.0cm,
    top=3.5cm,
    bottom=3.5cm
]{geometry}

\usepackage{txfonts}

\usepackage{mathrsfs}
\usepackage{color}
\usepackage[utf8]{inputenc}
\usepackage[dvipsnames]{xcolor}
\usepackage[colorlinks=true,urlcolor={YellowOrange},linkcolor={YellowOrange},citecolor={YellowOrange}]{hyperref}  
\usepackage{float}
\usepackage{amsfonts,amssymb,amsmath}
\usepackage{amsthm}    
\usepackage{url}
\usepackage{paralist}
\usepackage{tikz}
    \usetikzlibrary{decorations.markings,arrows,shapes.callouts}
\usepackage{rotating} 
\usepackage{verbatim}
\usepackage{tikz-cd}
\usepackage{mathtools}
\usepackage{dsfont}
\usepackage{amsrefs}
\usepackage{thmtools}
\usepackage{thm-restate}

\newtheorem{theorem}{Theorem}[section]
\newtheorem{proposition}[theorem]{Proposition}
\newtheorem{lemma}[theorem]{Lemma} 
\newtheorem{conjecture}[theorem]{Conjecture}

\newtheorem{corollary}[theorem]{Corollary} 

\theoremstyle{definition}
\newtheorem{definition}[theorem]{Definition}
\theoremstyle{remark}

\def\keywords{\xdef\@thefnmark{}\@footnotetext}

\tikzcdset{row sep/my_size/.initial=-2mm}
\tikzcdset{row sep/my_size_small/.initial=-3mm}
\newcommand{\xDownarrow}[1]{%
  {\left\Downarrow\vbox  to #1{}\right.\kern-\nulldelimiterspace}
}

\begin{document}

\title[]{Multi-graded generic initial ideals, regularity, and the optimal colorful fractional Helly theorem for $d$-Leray complexes}

\author[]{Daniel McGinnis} 
\address{Discrete Mathematics Group, Institute for Basic Science (IBS), Daejeon,
South Korea}
\email{dmcginnis@ibs.re.kr}




\begin{abstract}
    A celebrated result of Bayer and Stillman from 1987 states that for a homogeneous ideal $I$ of a polynomial ring $S$, the regularities of $S/I$ and $S/\textrm{GIN}(I)$ are the same under the reverse lexicographic monomial ordering, where $\textrm{GIN}(I)$ is the generic initial ideal. If $R$ is a polynomial ring whose variables are subdivided into disjoint blocks of variables $X_1,\dots,X_c$, there is  a natural multi-grading on $R$, and one can analogously define a multi-graded version of the generic initial ideal for any multi-homogeneous ideal $I$ of $R$. However, the full strength of the Bayer--Stillman Theorem fails in the multi-graded setting; there are multi-homogeneous ideals $I$ such that the regularities are not preserved after passing to the multi-graded generic initial ideal no matter the choice of monomial ordering.

    We prove lower bounds on the regularity of $R/I$ in terms of \textit{almost regular sequences} of the multi-graded generic initial ideal of $I$ restricted to each block of variables. Again, we use the reverse lexicographic monomial ordering, but interestingly, the lower bound result requires a particular choice of ordering on the variables. 
    
    As an application, we prove the \textit{optimal fractional Helly theorem for $d$-Leray simplicial complexes}, a problem stemming from the work of Kim in 2017.
\end{abstract}

\maketitle

\setcounter{tocdepth}{1}


\section{Introduction}

Let $S=\Bbbk [x_1,\dots,x_n]$ be a polynomial ring with the standard $\mathbb{N}$-grading (throughout $\mathbb{N}$ will be understood to contain $0$) over an infinite field $\Bbbk$. Given a homogeneous ideal $I$ of $S$, one can associate the corresponding monomial ideal $\textrm{in}(I)$ when a monomial order is specified on $S$. See  Eisenbud \cite[Chapter 15]{eisenbudBook} for a detailed account of monomial orders and initial ideals. Since monomial ideals are typically easier to work with, it is desirable to determine properties of $I$ that are preserved after taking the initial ideal. For instance, the Hilbert series of $I$ and $\textrm{in}(I)$ coincide.

An important invariant of a homogeneous ideal $I$ of $S$ is the \textit{graded Betti numbers} $\beta_{i,j}(S/I)$, defined by 
\[
\beta_{i,j} = \textrm{dim}_\Bbbk\textrm{Tor}_i(\Bbbk ,S/I)_j.
\]
An equivalent definition is that $\beta_{i,j}$ is the number of copies of $S(-j)$ in the $i$'th homological degree of the minimal free resolution of $S/I$. The graded Betti numbers can also be defined in the exact same way for finitely-generated $S$-modules, see Peeva \cite{PeevaBook} for details on this subject.

The \textit{regularity} of $S/I$ (or more generally a finitely-generated $S$-module), denoted by $\textrm{reg}(S/I)$, is $\max\{j \mid \beta_{i,i+j} \neq 0 \textrm{ for some }i\}$. In general, one has that $\textrm{reg}(S/I)\leq \textrm{reg}(S/\textrm{in}(I))$, however, under additional assumptions, we have equality. This follows from a celebrated result of Bayer and Stillman. We review the necessary terminology in Section \ref{sec:GIN and RevLex}.

\begin{theorem}[Bayer--Stillman \cite{Bayer1987Criterion}]\label{thm:BayerStillman}
    Let $I$ be a homogeneous ideal of $S$ with the reverse lexicographic monomial order. There is a Zariski open set $U\subset \textrm{GL}_{\Bbbk }(n)$ and an ideal $J$ such that $J=\textrm{in}(uI)$ for all $u\in U$, and $\textrm{reg}(J) = \textrm{reg}(I)$.
\end{theorem}

The ideal $J$ in Theorem \ref{thm:BayerStillman} is called the \textit{generic initial ideal} of $I$ and we denote it as $\textrm{GIN}(I)$. A strengthening of Theorem \ref{thm:BayerStillman} is that $I$ and $\textrm{GIN}(I)$ have the same \textit{extremal Betti numbers}, see  Definition \ref{def:ext betti} and Theorem \ref{thm:GINextremalBettEq}. 

Such results on homogeneous ideals have applications on the combinatorics of simplicial complexes via the Stanley-Reisner ring $S/I_\Delta$. Here $\Delta$ is a simplicial complex on the vertex set $[n]=\{1,\dots,n\}$, and $I_\Delta$ is the monomial ideal generated by monomials $x_{i_1}\dots,x_{i_k}$ where $\{i_1,\dots,i_k\}$ is not a face of $\Delta$. For example, algebraic (symmetric) shifting is a technique developed by Kalai that associates a simplicial complex $\Delta$ to a \textit{shifted complex} $\Delta^s$ that satisfies nice combinatorial properties (see Section \ref{sec:Alg Shifting}). For instance, $\Delta$ and $\Delta^s$ have the same face numbers. See also the survey \cite{Kalai2002AlgebraicShifting} for more information and history. Essentially, the construction is to apply the generic initial ideal to $I_\Delta$; we refer to the book \cite{herzog2011monomial} for a very nice detailed account of algebraic shifting. We also give a brief overview in Section \ref{sec:Alg Shifting}. Theorem \ref{thm:BayerStillman} can be used to show that $S/I_\Delta$ and $S/I_{\Delta^s}$ have the same regularity. This gives a connection between the homology of induced complexes of $\Delta$ and $\Delta^s$ via Hochster's formula:

\begin{theorem}[Hochster's Formula \cite{HochstersFormula}]\label{thm:HochstersFormula}

Let $\Delta$ be a simplicial complex on the vertex set $[n]$. Then
\[
\beta_{i,j}(S/I_\Delta) = \sum_{W\in {[n] \choose j}} \textrm{dim}_\Bbbk\tilde{H}_{j-i-1}(\Delta[W];\Bbbk ).
\]
    
\end{theorem}

A simplicial complex is said to be \textit{$d$-Leray} if $\tilde{H}_k(\Delta[W])=0$ for all $k\geq d$ and for all subsets of vertices $W$ (here we are taking homology with $\mathbb{Q}$ coefficients). By Hochster's formula, a simplicial complex is $d$-Leray if and only if its Stanley-Reisner ring is $d$-regular. Therefore, we have that $\Delta$ is $d$-Leray if and only if $\Delta^s$ is $d$-Leray. This, together with the fact that algebraic shifting preserves the face numbers,  can be used to prove the following upper bound theorem for $d$-Leray complexes, see Lemma \ref{lem: 1- color comb main} for a proof. Here, $f_k(\Delta)$ is the number of $k$-dimensional faces of $\Delta$.

\begin{theorem}\label{thm:d-leray-upbound}
    Let $\Delta$ be a $d$-Leray simplicial complex with vertex set $[n]$ and whose largest face has size $r$. Then $f_d(\Delta)\leq {n\choose d+1}-{n-r+d \choose d+1}$.
\end{theorem}

The motivation for Theorem \ref{thm:d-leray-upbound} comes from a conjecture of Katchalski and Perles whose statement is the same as Theorem \ref{thm:d-leray-upbound} except for complexes $\Delta$ that arise as the nerve of a family of convex sets in $\mathbb{R}^d$. Their conjecture also includes a bound on the number of $k$-dimensional faces for $k\geq d$. The upper bound is achieved by taking the nerve of the family of convex sets in $\mathbb{R}^d$ consisting of $r-d$ copies of $\mathbb{R}^d$ and $n-(r-d)$ hyperplanes in general position. This upper bound theorem for convex sets was proved independently by Kalai \cite{kalai1984intersection} and Eckhoff \cite{eckhoff1985Upper}, and the proof of Kalai initiated the method of algebraic shifting. By the Nerve Lemma, the nerve of a family of convex sets in $\mathbb{R}^d$ is, in particular, $d$-Leray. An explicit proof of Theorem \ref{thm:d-leray-upbound} using the algebraic shifting technique also appears in \cite{Bulavka2021Optimal}. 

One of the main consequences of Theorem \ref{thm:d-leray-upbound} is the \textit{optimal fractional Helly Theorem} for $d$-Leray complexes; this is a cornerstone result in discrete geometry when $\Delta$ comes from the nerve of a family of convex sets in $\mathbb{R}^d$. See \cite{amenta2017Helly} for a survey on Helly-type theorems.

\begin{theorem}[Optimal fractional Helly Theorem]\label{thm:frac helly}
    Let $\alpha \in (0,1]$, $d\geq 1$, and let $\Delta$ be a $d$-Leray simplicial complex on $n$ vertices. If $f_d(\Delta)\geq \alpha {n \choose d+1}$, then the largest face of $\Delta$ has size at least $(1-(1-\alpha)^{1/(d+1)})n$.
\end{theorem}

Theorem \ref{thm:frac helly} is optimal in the sense that the constant $(1-(1-\alpha)^{1/(d+1)})$ in front of $n$ is best possible. The original Helly theorem is the statement of Theorem \ref{thm:frac helly} when $\alpha=1$ and $\Delta$ is the nerve of a family of convex sets. The deduction from Theorem \ref{thm:d-leray-upbound} to Theorem \ref{thm:frac helly} can be found in \cite{kalai1984intersection}. We note that a (non-optimal) fractional Helly Theorem for convex sets was first proven in  \cite{katchalski1979Problem}.

The content of this manuscript is motivated by the following conjectured \textit{colorful} analogue of Theorem \ref{thm:d-leray-upbound}. Let $\Delta$ be a simplicial complex whose vertex set is the disjoint union $V_1\dot{\cup}\cdots\dot{\cup}V_c$. For a tuple $\mathbf{t}\in \mathbb{N}^c$, we define the flag face number $f_{\mathbf{t}}(\Delta)$ to be the number of faces of $\Delta$ that have exactly $\mathbf{t}_j$ vertices in $V_j$ for each $j$.

\begin{conjecture}[\cite{Bulavka2021Optimal}]\label{conj:Upper Main}
    Let $\Delta$ be a $d$-Leray simplicial complex on the vertex set $V_1\dot{\cup}\cdots \dot{\cup} V_{d+1}$. If $n_j=|V_j|$ and the largest face of $\Delta[V_j]$ has size at most $r_j$ for each $j$, then
    \[
    f_{(1,\dots,1)}(\Delta) \leq n_1\cdots n_{d+1} - (n_1-r_1)\cdots (n_{d+1}-r_{d+1}).
    \]
\end{conjecture}

The optimality of the upper bound in Conjecture \ref{conj:Upper Main} follows from a construction of Kim \cite{kim2017note}, where Conjecture \ref{conj:Upper Main} is implicit. Since then, Conjecture \ref{conj:Upper Main} has garnered notable interest in the community. For instance, the statement of the conjecture has been reiterated several times in various research and survey papers, \cite{barany2022helly,patak2026hellytype,Bulavka2021Optimal,chakraborti2026colorful}. Conjecture \ref{conj:Upper Main} is shown to be true in the case that $\Delta$ is the nerve of a family of convex sets in $\mathbb{R}^d$ in \cite{Bulavka2021Optimal} by adapting the method of Kalai in \cite{kalai1984intersection}, and another proof appears in \cite{chakraborti2026colorful}; actually it is proven more generally for the class of $d$-collapsible simplicial complexes, which is contained in the class of $d$-Leray complexes (see \cite{martin2013intersection}). Again, an important consequence of Conjecture \ref{conj:Upper Main} is an optimal \textit{colorful} fractional Helly theorem.

\begin{conjecture}[\cite{Bulavka2021Optimal}]\label{conj:col frac helly}

Let $\Delta$ be a $d$-Leray simplicial complex on the vertex set $V_1\dot{\cup}\cdots \dot{\cup} V_{d+1}$. Let $n_j=|V_j|$ and $\alpha \in (0,1]$ and assume that $f_{(1,\dots,1)}(\Delta) \geq \alpha n_1\cdots n_{d+1}.$ Then for some $j$, $\Delta[V_j]$ contains a face of size at least $(1-(1-\alpha)^{1/(d+1)})n_j$.
    
\end{conjecture}

\noindent The deduction from Conjecture \ref{conj:Upper Main} to Conjecture \ref{conj:col frac helly} can be found in \cite{Bulavka2021Optimal}. In Section \ref{sec:Comb main sec}, we prove Conjecture \ref{conj:Upper Main} and hence, Conjecture \ref{conj:col frac helly} holds as well. The original colorful Helly theorem is the statement of Conjecture \ref{conj:col frac helly} when $\alpha=1$ and $\Delta$ is the nerve of a family of convex sets. The question of whether the variants of Helly's Theorem extend to a $d$-Leray version is a fruitful direction in the area. For instance, the colorful Helly theorem for $d$-Leray complexes is shown to hold in \cite{kalai2005topological}, and it has proven to be an influential result within the community. See also \cite{holmsen2016intersection} for a further generalization of this result. Thus, a solution to Conjecture \ref{conj:Upper Main}, being a common generalization to the colorful Helly theorem for $d$-Leray complexes and Theorem \ref{thm:d-leray-upbound}, is a natural next step in this line of research.

Since algebraic shifting can be used to prove Theorem \ref{thm:d-leray-upbound}, one may hope to prove Conjecture \ref{conj:Upper Main} by using a colorful analogue. \textit{Colored algebraic shifting} was defined in \cite{Babson2004Face} for simplicial complexes $\Delta$ on a vertex set $V_1\dot{\cup}\cdots \dot{\cup} V_{c}$. Colored algebraic shifting preserves the flag face numbers, but an example from \cite{Murai2008Betti} shows that it does not preserve the $d$-Leray property in general; the difficulties of applying colored algebraic shifting are similarly noted in \cite{Bulavka2021Optimal}. Thus, one would have to show a more subtle relationship between $\Delta$ and its colored algebraic shifting to prove Conjecture \ref{conj:Upper Main} using this technique. This is one of the main goals of this paper.

When working in this colored setting, we work over a multi-graded polynomial ring. Let 
\[
X_j=\{x_{1,j},\dots,x_{n_j,j} \}
\]
be sets of variables with the multi-grading $(0,\dots,1,\dots,0)\in \mathbb{N}^c$ (where $1$ is in the $j$'th position) for $1\leq j\leq c$. We often identify the variables with the vertices of $\Delta$. Let $R=\Bbbk [\bigcup_{j=1}^c X_j]$ be the multi-graded polynomial ring. Throughout, we will always take the reverse lexicographic monomial order (see Section \ref{sec:GIN and RevLex}) and in each block of variables, we always impose $x_{1,j} > \cdots > x_{n_j,j}$. We will impose further conditions on the ordering of the variables throughout, and we make it explicit when we do. Given a multi-homogeneous ideal $I\subset R$, we can apply a generic block diagonal matrix $u\in G=\textrm{GL}_{\Bbbk }(n_1)\times \cdots \times \textrm{GL}_{\Bbbk }(n_c)$ (see Theorem \ref{thm:color zariski open}) and take the initial ideal to obtain the \textit{$G$-generic initial ideal} $G\textrm{-GIN}(I)=\textrm{in}(uI)$. Note that $\textrm{in}(uI)$ is still multi-homogeneous. The colored algebraic shifting of $\Delta$, which we call $\Delta^{cs}$, is essentially obtained by applying the $G$-generic initial ideal to the Stanley-Reisner ideal $I_\Delta$ (see Section \ref{sec:Alg Shifting} for more details). 

The reason why colored algebraic shifting can not immediately be applied to prove Conjecture \ref{conj:Upper Main} using the previously known properties of colored algebraic shifting in an analogous way that algebraic shifting can be used to prove Theorem \ref{thm:d-leray-upbound}, is essentially due to the fact that the full strength of the Bayer--Stillman Theorem does not hold. This is why, as previously mentioned, the colored algebraic shifting of a $d$-Leray simplicial complex may not be $d$-Leray. Most attempts at generalizing the Bayer--Stillman Theorem in the multi-graded setting involve a notion of multi-graded Betti number and multi-graded regularity. The paper \cite{maclagan2004multigraded} is one of the first major contributions in this direction. For more recent developments in this vein, see \cite{bender2026bigraded}, and the references therein. Unfortunately, the known results of this form do not seem to have any bearing on Conjecture \ref{conj:Upper Main}.  However, in this paper, we only work with the usual notions of graded Betti numbers and regularity. One of the main results of this paper is to give a lower bound on the regularity of $R/I$ based on information from $R/G\textrm{-GIN}(I)$ (see Theorem \ref{thm:Alg Main}).

\subsection{Main results}

We present the main results of the paper here. Recall that $R=\Bbbk [\bigcup_{j=1}^c X_j]$ is a multi-graded polynomial ring. Our main algebraic result provides a lower bound on the regularity of $R/I$. For a multi-homogeneous ideal $I\subset R$ and an index $1\leq j\leq c$, we set $I_j=I\cap \Bbbk[X_j]$. For a finite length $R$-module $N$, we define $s(N)$ to be its top degree. For a multi-graded $R$-module, there is a corresponding $\mathbb{N}$-grading; it will be clear from context when we refer to the multi-grading or the $\mathbb{N}$-grading. The $G$-generic initial ideal here is taken with respect to the reverse lexicographic monomial order corresponding to a particular ordering on the variables, detailed in Section \ref{sec:multi GIN}.


\begin{restatable}{theorem}{AlgMain}\label{thm:Alg Main}
     Suppose $1\leq j_1 <\cdots < j_k\leq c$ and $d_{1,j_t}\geq \cdots \geq d_{i_{j_t},j_t}\geq 1 $ are integers for $1\leq t\leq k$. Take the reverse lexicographic monomial order where the variables $\{x_{1,{j_t}},\dots,x_{i_{j_t},{j_t}} \mid 1\leq t \leq k\}$ are greater than any variable in $\{x_{i_{j_t}+1,{j_t}},\dots,x_{n_{j_t},{j_t}} \mid 1\leq t \leq k\}$. Suppose $I\subset R$ is a multi-homogeneous ideal, and let $J=G\textrm{-GIN}(I)$.  Furthermore, assume that there exist monomials 
    \[
    a_t\in \left(J_{j_t} + (x_{i_{j_t}+1,j_t},\dots,x_{n_{j_t},j_t}):_{R_{j_t}} x_{i_{j_t},j_t}^{d_{i_{j_t},j_t}}\right)_{q_t}
    \]
    such that $a_1\cdots a_k\notin J+\left(x_{i_{j_t}+1,j_t},\dots,x_{n_{j_t},j_t} \mid 1\leq t\leq k\right)$ and $q_{t}+d_{i_{j_t}-1,j_t} > \textrm{reg}(R_{j_t}/J_{j_t})=\textrm{reg}(R_{j_t}/I_{j_t})$ if $i_{j_t}\geq 2$ for  $1\leq t\leq k$.
    Then 
    \[
    \textrm{reg}(R/I)\geq q_1+\cdots+q_k. 
    \] 
\end{restatable}

Theorem \ref{thm:Alg Main} can be used to prove certain restrictions on the possible faces of $\Delta^{cs}$ when $\Delta$ is $d$-Leray. This is the content of Theorem \ref{thm:comb main} below. For a $c$-tuple $\mathbf{t}\in \mathbb{N}^c$, let $X_{\mathbf{t},j}=\{x_{n_j-r_j+\min\{\mathbf{t}_j,d+1\},j},\dots,x_{n_j,j}\}$. For a subset $F$ of the vertex set of $\Delta$, define 
\[
g_{\mathbf{t},j}(F)=\begin{cases}
|F\cap (X_j\setminus{X_{\mathbf{t},j}})| & \textrm{if } |F\cap (X_j\setminus{X_{\mathbf{t},j}})|\geq \min\{\mathbf{t}_j,d+1\},\\
 0 &\textrm {otherwise.}
\end{cases}
\]
Throughout, ${V \choose \mathbf{t}}$ denotes the subsets of $V$ with exactly $\mathbf{t}_j$ vertices in $V_j$ for all $j$.

\begin{restatable}{theorem}{CombMain}\label{thm:comb main}
	Let $\Delta$ be a $d$-Leray simplicial complex with vertex set $V=V_1\dot{\cup} \cdots \dot{\cup} V_c$. Let $n_j=|V_j|$ and assume the largest face of $\Delta[V_j]$ has size at most $r_j$ for each $j$. If  $F\in {V \choose \mathbf{t}}$ is a set such that $\sum g_{\mathbf{t},j}(F) > d$, then $F$ is not a face of $\Delta^{cs}$ with respect to the reverse lexicographic monomial order where the variables in $\{x_{1,j},\dots,x_{n_{j}-r_{j},j}\mid 1\leq j\leq c \}$ are greater than any variable in $\{x_{n_{j}-r_{j}+1,j},\dots,x_{n_{j},j}\mid 1\leq j\leq c \}$.
\end{restatable}

A direct application of Theorem \ref{thm:comb main} then proves Conjecture \ref{conj:Upper Main} and hence also Conjecture \ref{conj:col frac helly}.

\begin{restatable}{theorem}{ConjecturesHold}
    Conjectures \ref{conj:Upper Main} and \ref{conj:col frac helly} hold.
\end{restatable}

We can obtain a more general upper bound result again as a direct application of Theorem \ref{thm:comb main}. For $\mathbf{n}=(n_1,\dots,n_c)$, $\mathbf{r}=(r_1,\dots,r_c)$ and a tuple $\mathbf{t}\in \mathbb{N}^c$, define 
\[
\mathscr{F}_{\mathbf{t}}(\mathbf{n},d,\mathbf{r}) = \left\{ F\in {V\choose \mathbf{t}} \mid \sum g_{\mathbf{t},j}(F) > d\right\}.
\]

\begin{corollary}\label{cor:gen Upper Bound}
        Let $\Delta$ be a $d$-Leray simplicial complex with vertex set $V=V_1\dot{\cup} \cdots \dot{\cup} V_c$. Let $n_j=|V_j|$ and assume the largest face of $\Delta[V_j]$ has size at most $r_j$ for each $j$. Then
        \[
        f_{\mathbf{t}}(\Delta) \leq \left|\left\{F\in {V\choose \mathbf{t}} \mid F \textrm{ contains no set in } \mathscr{F}_{\mathbf{k}}(\mathbf{n},d,\mathbf{r}) \textrm{ for any }\mathbf{k}\in \mathbb{N}^c  \right\}\right|.
        \]
\end{corollary}
\noindent Furthermore, we show in Proposition \ref{prop:Upper Bound Tight} that the simplicial complex defined by 
\[
\left\{F\subseteq {V} \mid F \textrm{ contains no set in } \mathscr{F}_{\mathbf{k}}(\mathbf{n},d,\mathbf{r}) \textrm{ for any }\mathbf{k}\in \mathbb{N}^c  \right\}
\]
is itself $d$-Leray (actually we show that it is $d$-collapsible), which implies that the bound in Corollary \ref{cor:gen Upper Bound} is tight.

Another incredible fact about algebraic shifting is that it preserves the topological Betti numbers of $\Delta$: Define $\beta_i(\Delta) = \textrm{dim}_{\mathbb{Q}}\tilde{H}_i(\Delta)$, where homology is taken with coefficients in $\mathbb{Q}$, then $\beta_i(\Delta) = \beta_i(\Delta^s)$ for $i\geq -1$. Additionally, one can read off the value of $\beta_i(\Delta)$ from the face structure of $\Delta^s$:
\[
\beta_{i-1}(\Delta) = \left| \{ F\in \Delta^s \mid |F| = i,\, F\cup \{n\} \notin \Delta^s\} \right|.
\]
For colored algebraic shifting, we have a lower-bound analogue of this result below. For a tuple $\mathbf{t}\in \mathbb{N}^c$, we impose further the following condition on the ordering of the variables: the variables in $\{x_{1,j},\dots,x_{n_j-\mathbf{t}_j,j} \}$ are greater than each variable in $\{x_{n_j-\mathbf{t}_j+1,j},\dots,x_{n_j,j} \}$ for each $j$. We fix a monomial order $>_\mathbf{t}$ that satisfies this ordering condition on the variables, and we define $\Delta_{\mathbf{t}}^{cs}$ to be the colored algebraic shifting with respect to this order.

\begin{restatable}{theorem}{bettiLowBound}\label{thm:Betti lower bound}
Let $\Delta$ be a simplicial complex with vertex set $V=V_1\dot{\cup} \cdots \dot{\cup} V_c$. Then
\[
\beta_{i-1}(\Delta) \geq \sum_{\mathbf{t}}\left|\{ F\in \Delta_{\mathbf{t}}^{cs} \cap {V \choose \mathbf{t}} \mid (F\cap V_j)\cup \{x_{n_j,j}\}\notin \Delta_{\mathbf{t}}^{cs} \textrm{ for all } j \} \right|,
\]
where the sum ranges over tuples $\mathbf{t}\in \mathbb{N}^c$ such that $\sum_{j=1}^c\mathbf{t}_j=i$.
\end{restatable}

\section{Almost regular sequences and preliminary results}

A key notion is that of an \textit{almost regular sequence}, originally defined in \cite{aramova2000Almost}. We briefly review the relevant definitions and results on this topic. We note that the definition of an almost regular sequence varies slightly here in that we do not require the sequence to consist of linear forms.

\begin{definition}
    Let $\ell_1,\dots,\ell_n$ be a sequence of homogeneous elements in $S$. Let $M$ be a finitely-generated graded $S$-module, and define $M\langle j\rangle=M/(\ell_1,\dots,\ell_j)M$. Then $\ell_1,\dots,\ell_n$ is called an \textit{almost regular sequence} of $M$ if it is a regular sequence of $S$ and $(0:_{M\langle j-1 \rangle} \ell_j)$ is a module of finite length for all $1\leq j\leq n$. 
\end{definition}

Let $M$ be a finitely-generated $S$-module with the almost regular sequence $\ell_1,\dots,\ell_n$ with degrees $k_1,\dots,k_n$, respectively. Consider the Koszul complex $K_{\bullet}(\ell_1,\dots,\ell_j;M)$, and let $H_i(j)=H_{i}(\ell_1,\dots,\ell_j;M)$ be the $i$'th Koszul homology. Additionally, for a finite length graded $S$-module $N$, we set $s(N)$ to be the top degree of $N$. We define
\begin{align*}
    m_j =\textrm{max}\{s(H_i(j)) - (k_j+\cdots+k_{j-i+1})\mid 1\leq i\leq j \}\ \textrm{ and }\ s_j = s\left((0:_{M\langle j-1 \rangle} \ell_j) \right)
\end{align*}
for $j=1,\dots,n$, and we additionally set $m_0=0$ and $s_0=0$ (we will see that $H_i(j)$ has finite length, so $s(H_i(j))$ is well-defined). The following result is a slight generalization of Theorem 1.1 in \cite{aramova2000Almost}. The proof is essentially the same as well, but we include it here for completeness.

We use the well-known fact (see for instance \cite{bruns1993cohen}) that there is a long exact sequence
    \begin{align}\label{eq:LES}
       \cdots &\longrightarrow H_{i+1}(j)\longrightarrow H_i(j-1)(-k_j)\longrightarrow H_i(j-1) \longrightarrow H_i(j)\longrightarrow \cdots\\
       &\longrightarrow H_2(j)\longrightarrow H_1(j-1)(-k_j) \longrightarrow H_1(j-1) \longrightarrow H_1(j) \longrightarrow (0:_{M\langle j-1 \rangle} \ell_j)(-k_j)\longrightarrow 0 \nonumber,
    \end{align}
    where the map $H_i(j-1)(-k_j)\longrightarrow H_i(j-1)$ is multiplication by $\ell_j$. This arises from the short exact sequence of complexes given in each homological degree by
    \[
    0\longrightarrow K_i(\ell_1,\dots,\ell_{j-1};M) \longrightarrow K_i(\ell_1,\dots,\ell_j;M) \longrightarrow K_{i-1}(\ell_1,\dots,\ell_{j-1};M)(-k_j)\longrightarrow 0.
    \]
\begin{theorem}\label{thm:Ext Betti Gen}
    Let $\ell_1,\dots,\ell_n$ be an almost regular sequence of $M$ with degrees $k_1\leq \cdots \leq k_n$. Then
    \begin{itemize}
        \item[(a)] $m_j = \textrm{max}\{ s_1,\dots,s_j\}$ for $j=1,\dots,n$. In particular, $m_1\leq \cdots \leq m_n$.
        \item[(b)] Let $i_1<\cdots<i_p$ be the indices $i$ such that $m_i-m_{i-1} > 0$. Then for all $1\leq t \leq p$ and all $j\geq i_t$ we have:
        \begin{itemize}
            \item[(i)] $H_i(j)_{k_j+\cdots+k_{j-i+1}+s}=0$ for $s>m_{i_{t-1}}$ and $i>j-i_t+1$,
            \item[(ii)] $H_{j-i_t+1}(j)_{k_j+\cdots+ k_{i_t}+m_{i_t}} \cong (0:_{M\langle i_t -1 \rangle} \ell_{i_t})_{s_{i_t}}$.
        \end{itemize}
    \end{itemize}
\end{theorem}
\begin{proof}
     For (a), we show by induction on $j$ that $m_j=\textrm{max}\{m_{j-1},s_j\}$. For $j=1$ we have $H_i(1)=0$ for $i>1$ and $H_1(1) \cong (0:_M \ell_1)(-k_1)$, so $m_1 = s_1$ as desired.

    Let $j>1$ and assume that $s>\textrm{max}\{m_{j-1},s_j \}$. From the long exact sequence, we have the exact sequence
    \begin{equation}\label{eq:1st exact Sequence}
        H_i(j-1)_{k_j+\cdots + k_{j-i+1} + s} \longrightarrow H_i(j)_{k_j+\cdots + k_{j-i+1} + s} \longrightarrow H_{i-1}(j-1)_{k_{j-1}+\cdots + k_{j-i+1} + s}
    \end{equation}
    for all $i\geq 1$ (when $i=1$, we take $H_{i-1}(j-1) = (0:_{M\langle j-1 \rangle} \ell_j)$). Working inductively, the left and right homology groups are $0$ (for the left homology group, we are using the fact that $k_j\geq k_{j-i}$), so $H_i(j)_{k_j+\cdots + k_{j-i+1} + s}=0$ as well. 

    Now suppose that $m_j < \textrm{max}\{m_{j-1},s_j\}$. In the case that $m_j < m_{j-1}$, let $i\geq 1$ be an integer such that $H_i(j-1)_{k_{j-1}+\cdots+k_{j-i}+m_{j-1}}\neq 0$. From the long exact sequence, we have the exact sequence
    \[
    H_{i+1}(j)_{k_{j}+\cdots+k_{j-(i+1)+1}+m_{j-1}} \longrightarrow H_i(j-1)_{k_{j-1}+\cdots+k_{j-i}+m_{j-1}}\longrightarrow H_i(j-1)_{k_{j}+\cdots+k_{(j-1)-i+1}+m_{j-1}}.
    \]
    This is a contradiction since, again, the left and right homology groups are $0$.

    In the case $m_j<s_j$, the exact sequence
    \[
    H_1(j)_{k_j+s_j} \longrightarrow (0:_{M\langle j-1 \rangle} \ell_j)_{s_j} \longrightarrow 0
    \]
    implies that $(0:_{M\langle j-1 \rangle} \ell_j)_{s_j}=0$, a contradiction.

    For (b)(i), we start with the case $j=i_t$ and work inductively. We have $j-1 < i_t$, so $m_{j-1} = m_{i_{t-1}}$. Therefore, $H_i(j-1)_{k_{j-1} + \cdots k_{j-i}+s}=0$ for $s>m_{i_{t-1}}$ and $i\geq 1$ by (a). The sequence (\ref{eq:1st exact Sequence}) then implies that $H_i(j)_{k_j+\cdots k_{j-i+1} + s}=0$ for $s>m_{i_{t-1}}$ and $i>1=j-i_t+1$. When $j>i_t$, $j-1\geq i_t$, and we have by induction that $H_i(j-1)_{k_{j} + \cdots + k_{j-i+1}+s}=H_{i-1}(j-1)_{k_{j-1} + \cdots + k_{j-i+1}+s}=0$ when $s>m_{i_{t-1}}$ and $i>j-i_t+1$.

    For (b)(ii), first note that we have the exact sequence
    \[
    H_1(i_t-1)_{k_{i_t}+m_{i_t}} \longrightarrow H_1(i_t)_{k_{i_t}+m_{i_t}} \longrightarrow (0:_{M\langle i_t-1 \rangle} \ell_{i_t})_{m_{i_t}}\longrightarrow 0.
    \]
    Since the left homology group is $0$, (b)(ii) then follows for $j=i_t$.

    Now suppose $j > i_t$, and consider the exact sequence
    \begin{align*}
    &H_{j-i_t+1}(j-1)_{k_j+\cdots + k_{i_t} + m_{i_t}} \longrightarrow H_{j-i_t+1}(j)_{k_j+\cdots + k_{i_t} + m_{i_t}} \longrightarrow H_{j-i_t}(j-1)_{k_{j-1}+\cdots + k_{i_t} + m_{i_t}}\\
    &\longrightarrow H_{j-i_t}(j-1)_{k_j+\cdots + k_{i_t} + m_{i_t}}.
    \end{align*}
    It follows directly from (b)(i) that the left homology group is $0$. The right homology group is also $0$: if $j-1<i_{t+1}$, then $m_{j-1}=m_{i_t}$ and it follows from the definition of $m_{j-1}$; if $j-1\geq i_{t+1}$, it follows directly from (b)(i). Thus, $H_{j-i_t+1}(j)_{k_j+\cdots+k_{i_t}+m_{i_t}} \cong H_{j-i_t}(j-1)_{k_{j-1}+\cdots + k_{i_t} + m_{i_t}}$, which completes the inductive proof.
\end{proof}

In particular, Theorem \ref{thm:Ext Betti Gen} implies that the regularity of $M$ is given by $\textrm{max}\{s_1,\dots,s_n,s(M/\mathfrak{m}M)\}$ ($\mathfrak{m}$ represents the irrelevant ideal) when $\ell_1,\dots,\ell_n$ is an almost regular sequence of linear forms. This is due to the fact that when $\ell_1,\dots,\ell_n$ is an almost regular sequence of linear forms, $H_i(n)=\textrm{Tor}_i(\Bbbk,M)$. In Proposition \ref{prop:Reg Lower} below, we show a more general way to obtain lower bounds on regularity. Recall that $S/(\ell_1,\dots,\ell_n)$ has finite length since $\ell_1,\dots,\ell_n$ is by definition a regular sequence of $S$.

\begin{proposition}\label{prop:Reg Lower}
    Let $M$ be a finitely-generated graded $S$-module, $\ell_1,\dots,\ell_n$ be an almost regular sequence of $M$, and $m$  the top degree of $S/(\ell_1,\dots,\ell_n)$. If $N=\textrm{max}\{j-i \mid \textrm{Tor}_i(S/(\ell_1,\dots,\ell_n),M)_j \neq 0 \}$, then $\textrm{reg}(M) \geq N-m$.
\end{proposition}
\begin{proof}
    Let $F_{\bullet}$ be the minimal free resolution of $M$, and let $j$ and $i$ be indices that witness $N$. We have that $\textrm{Tor}_i(S/(\ell_1,\dots,\ell_n), M)_j = H_i(S/(\ell_1,\dots,\ell_n)\otimes F_{\bullet})_j \neq 0$. Since the top degree of $S/(\ell_1,\dots,\ell_n)$ is $m$, we must have that the free module $F_i$ has a generator of degree at least $j-m$. Therefore $\textrm{reg}(M) \geq j-m-i=N-m$.
\end{proof}

\begin{corollary}\label{cor:Reg Lower}
    Let $M$ be a finitely-generated graded $S$-module, and let $\ell_1,\dots,\ell_n$ be an almost regular sequence of $M$ with degrees $k_1\leq \cdots \leq k_n$. Then $\textrm{reg}(M) \geq \textrm{max}\{s_1,\dots,s_n\}$.
\end{corollary}
\begin{proof}
Let $i$ be the index such that $s_i=\textrm{max}\{s_1,\dots,s_n\}$. By Theorem \ref{thm:Ext Betti Gen}, $H_{n-i+1}(n)_{k_n+\cdots +k_i +s_i}\neq 0$. We also have that the top degree of $S/(\ell_1,\dots,\ell_n)$ is $\sum_{i=1}^n(k_i-1)$; this can be seen for instance by computing its Hilbert series, using the fact that $\ell_1,\dots,\ell_n$ is a regular sequence of $S$. Also, $\textrm{Tor}_{n-i+1}(S/(\ell_1,\dots,\ell_n),M)_{k_n+\cdots +k_i +s_i} \cong H_{n-i+1}(n)_{k_n+\cdots +k_i +s_i}\neq 0$. Therefore, $\textrm{reg}(M)\geq s_i$ by Proposition \ref{prop:Reg Lower}.
\end{proof}

 Let $\textbf{e}_A$ for subsets $A$ of $\{1,\dots,n\}$ be the standard basis elements in the Koszul complex $K_{\bullet}(\ell_1,\dots,\ell_n)$ corresponding to the sequence $\ell_1,\dots,\ell_n$. In the following Proposition, we give a relationship between elements of $(0:_{M\langle i-1\rangle} \ell_{i})_{q}$ and elements of $H_{n-i+1}(n)_{k_n+\cdots+k_i+q}$ when $q +k_{i+1} > \textrm{max}\{s_1,\dots,s_{n}\}$ when $i\leq n-1$.

\begin{proposition}\label{prop:Kosz Hom rep}
Let $\ell_1,\dots, \ell_n$ be an almost regular sequence of a finitely-generated $S$-module $M$ with degrees $k_1\leq \cdots\leq k_n$. Assume ${a}\in ((\ell_1,\dots,\ell_{i-1})M:_M \ell_{i})_{q}\setminus (\ell_1,\dots,\ell_{i-1})M$ such that $q+k_{i+1} > \max\{ s_{1},\dots,s_n\} $ if $i\leq n-1$. There exists a nonzero element of $H_{n-i+1}(n)_{k_n+\cdots+k_i+q}$ represented by an element of the form $(a\mathbf{e}_{\{i,i+1,\dots,n \}}+\cdots)\in K_{n-i+1}(\ell_1,\dots,\ell_n;M)$.
\end{proposition}
\begin{proof}
    We prove via induction that there is an element represented by $(a\mathbf{e}_{\{i,i+1,\dots,j\}}+\cdots)$ in $H_{j-i+1}(j)_{k_j+\cdots+k_i + q}$ for $j\geq i$. For $j=i$, the fact that ${a}\in ((\ell_1,\dots,\ell_{i-1})M:_M \ell_{i})_{q}$ directly yields an element in $H_{1}(i)_{k_i + q}$. Indeed,  $a\ell_i+y_{i-1}\ell_{i-1} +\cdots+y_1\ell_1=0$ for some $y_{i-1},\dots,y_1\in M$, hence $a\mathbf{e}_{\{i\}}+\sum_{u=1}^{i-1} y_u\mathbf{e}_{\{u\}}$ represents a nonzero element in $H_{1}(i)_{k_i + q}$. 

    For the inductive step, let $(a\mathbf{e}_{\{i,i+1,\dots,j-1\}}+\cdots)$ represent an element in $H_{j-i}(j-1)_{k_{j-1}+\cdots+k_i + q}$. We get from the long exact sequence (\ref{eq:LES}) the exact sequence
    \[
    H_{j-i+1}(j)_{k_j+\cdots+k_i+q}\longrightarrow H_{(j-1)-i+1}(j-1)_{k_{j-1}+\cdots+k_i+q}\longrightarrow H_{(j-1)-i+1}(j-1)_{k_j+\cdots+k_i+q}.
    \]
    The homology group on the right is zero because $q+k_j>m_{j-1}$ by Theorem \ref{thm:Ext Betti Gen} (a). This implies that there is an element $b\in K_{j-i+1}(\ell_1,\dots,\ell_{j-1};M)$ whose boundary is $\ell_j(a\mathbf{e}_{\{i,i+1,\dots,j-1\}}+\cdots)$. By a diagram chase, an element of the form $(a\mathbf{e}_{\{i,i+1,\dots,j\}}+\cdots) \pm b$ represents an element of $H_{j-i+1}(j)_{k_j+\cdots+k_i+q}$. If $(a\mathbf{e}_{\{i,i+1,\dots,j\}}+\cdots) \pm b$ is a boundary in the Koszul complex, then $a\in (\ell_1,\dots,\ell_{i-1})M$, a contradiction. Therefore, $(a\mathbf{e}_{\{i,i+1,\dots,j\}}+\cdots) \pm b$ must represent a nonzero element of $H_{j-i+1}(j)_{k_j+\cdots+k_i+q}$. This completes the proof.
\end{proof}

\noindent We have a slightly stronger statement when $s_i>\max\{s_1,\dots,s_{i-1}\}$.

\begin{proposition}\label{prop:Kosz Hom iso}
Let $\ell_1,\dots, \ell_n$ be an almost regular sequence of a finitely-generated $S$-module $M$ with degrees $k_1\leq \cdots\leq k_n$. Assume that $s_i > \max\{ s_{1},\dots,s_{i-1}\}$. There is an isomorphism from $(0:_{M\langle i-1\rangle} \ell_{i})_{s_i}$ to  $H_{n-i+1}(n)_{k_n+\cdots+k_i+s_i}$ that sends $\bar{a}\in(0:_{M\langle i-1\rangle} \ell_{i})_{s_i}$ to an element represented by $(a\mathbf{e}_{\{i,i+1,\dots,n \}}+\cdots)\in K_{n-i+1}(\ell_1,\dots,\ell_n;M)$ where $a$ is a representative of $\bar{a}$.
\end{proposition}
\begin{proof}
    The same argument as the proof of Proposition \ref{prop:Kosz Hom rep} can be used to show that for each $\bar{a}\in (0:_{M\langle i-1\rangle} \ell_{i})_{s_i}$, there is an element of the form $(a\mathbf{e}_{\{i,i+1,\dots,n \}}+\cdots)$ in $K_{n-i+1}(\ell_1,\dots,\ell_n;M)$ so long as we can verify that the right-side homology group, $H_{(j-1)-i+1}(j-1)_{k_j+\cdots+k_i+s_i}$, from the exact sequence in the proof of Proposition \ref{prop:Kosz Hom rep} vanishes. If $s_i=\max\{s_1,\dots,s_n \}$, then this follows as in the proof of Proposition \ref{prop:Kosz Hom rep}. Otherwise, let $i_t>i$ be the first index such that $s_{i_t}>s_i$. If $j-1\geq i_t$, then the homology group is $0$ by Theorem \ref{thm:Ext Betti Gen} (b)(i). If $j-1 < i_t$, then the homology group is zero by Theorem \ref{thm:Ext Betti Gen} (a).

    Now choose a basis $\bar{a}_1,\dots,\bar{a}_k$ for the vector space $(0:_{M\langle i-1\rangle} \ell_{i})_{s_i}$; we define a vector space map by sending $\bar{a}_i$ to an element of $H_{n-i+1}(n)_{k_n+\cdots+k_i+s_i}$ represented by $(a_i\mathbf{e}_{\{i,i+1,\dots,n \}}+\cdots)$. This map has a trivial kernel because if an element $(a'\mathbf{e}_{\{i,i+1,\dots,n \}}+\cdots)$ is a boundary in $K_{n-i+1}(\ell_1,\dots,\ell_n;M)$, then $a'\in (\ell_1,\dots,\ell_{i-1})M$ and hence represents zero in $(0:_{M\langle i-1\rangle} \ell_{i})_{s_i}$.

    Therefore, it follows from Theorem \ref{thm:Ext Betti Gen} (b)(ii) that the map defined is an isomorphism.
\end{proof}

\subsection{Extremal Betti numbers}\label{sec:ext betti numbers}

The extremal Betti numbers defined below were introduced in \cite{bayer1999extremal}. They represent the values in the extremal ``corners'' of the Betti diagram of a finitely-generated $S$-module.

\begin{definition}\label{def:ext betti}
    Let $M$ be a finitely-generated graded $S$-module. The \textit{extremal Betti numbers} of $M$ are the Betti numbers $\beta_{i,i+j}$ such that for any $i'\geq i$ and $j'\geq j$, with at least one of the inequalities being strict, $\beta_{i',i'+j'} = 0$. We call $(i,i+j)$ an \textit{extremal Betti position} of $M$.
\end{definition}

The following corollary of Theorem \ref{thm:Ext Betti Gen} was noted in \cite{aramova2000Almost} although it is implicit in \cite{bayer1999extremal}.

\begin{corollary}[\cite{aramova2000Almost}]
    Let $\ell_1,\dots,\ell_n$ be an almost regular sequence of linear forms for a finitely-generated $S$-module $M$. Let $i_1<\cdots<i_k$ be the indices as in Theorem \ref{thm:Ext Betti Gen}. Then the extremal Betti numbers $\beta_{i,i+j}(M)$ for $i\geq 1$,  are precisely $\beta_{n-i_k+1,n-i_k+1+s_{i_k}}(M),\dots,\beta_{n-i_1+1,n-i_1+1+s_1}(M)$. Furthermore, $\beta_{n-i_t+1,n-i_t+1+s_{i_t}}(M) = \textrm{dim}_\Bbbk(0:_{M\langle i_{t-1} \rangle}\ell_{i_t})_{s_{i_t}}$.
\end{corollary}

\subsection{The generic initial ideal and reverse lexicographic order}\label{sec:GIN and RevLex}

Let $>$ be a monomial order on $S$, and let $I$ be a homogeneous ideal of $S$. Let $GL_\Bbbk (n)$ denote the invertible $n\times n$ matrices with entries in $\Bbbk $. For $u\in GL_\Bbbk (n)$, $uI$ is given by a linear change of coordinates determined by $u$ applied to $I$. One can consult \cite[Chapter 15]{eisenbudBook} for more details on the definitions and results of this section.

The \textit{generic initial ideal} of $I$, denoted $\textrm{GIN}(I)$, is the monomial ideal $\textrm{in}_>(uI)$ for a generic matrix $u$. The meaning of the term ``generic matrix $u$'' is made more precise by the following result.

\begin{theorem}\label{thm:GIN Open}
    Let $I$ be a homogeneous ideal of $S$, and let $>$ be a monomial order on $S$. There exists a Zariski open subset $U$ of $\textrm{GL}_\Bbbk (n)$ and a monomial ideal $J=\textrm{GIN}(I)$ such that $\textrm{in}_>(uI)=J$ for all $u\in U$.
\end{theorem}

The generic initial ideal of a homogeneous ideal $I$ satisfies the following useful property when the field $\Bbbk $ has characteristic $0$.

\begin{theorem}\label{thm:GIN StronglyStable}
    Let $I$ be a homogeneous ideal of $S$, and fix a monomial order on $S$. Additionally, assume the field $\Bbbk $ has characteristic $0$. Then $\textrm{GIN}(I)$ satisfies the following property:
    \begin{equation}\label{eq:StronglyStable}
        x_im \textrm{ is a monomial in $\textrm{GIN}(I)$ }\implies x_{i'}m\in \textrm{GIN}(I) \textrm{ for all $i'<i$}.
    \end{equation}
\end{theorem}

A monomial ideal satisfying (\ref{eq:StronglyStable}) is said to be \textit{strongly stable}. We observe below that monomial ideals $I$ with the strongly stable property have the property that $x_n,x_{n-1},\dots,x_1$ is an almost regular sequence of $S/I$. A similar statement holds when $\Bbbk$ has positive characteristic (\cite[Chapter 15]{eisenbudBook}), however, the characteristic $0$ case is the most relevant in the context of algebraic shifting, so we focus on this case.

\begin{proposition}\label{prop:SS AlmostReg}
    If $I$ is a strongly stable monomial ideal, then $x_n,x_{n-1},\dots,x_1$ is an almost regular sequence for $S/I$.
\end{proposition}
\begin{proof}
    The statement follows from the observation that the module 
    \[
    (0:_{S/(I+ (x_{i+1},\dots,x_n)} x_i)
    \]
    is annihilated by the maximal ideal $(x_1,\dots,x_n)$ due to the strongly stable property.
\end{proof}

The monomial order that we will primarily use is the \textit{reverse lexicographic order}, which has particularly nice properties that we will review. In the reverse lexicographic order, we have $x^\mathbf{a} > x^\mathbf{b}$ if the degree of $x^\mathbf{a}$ is higher than $x^\mathbf{b}$, or if they have the same degree, $i=\textrm{max}\{j \mid \mathbf{a}_j \neq \mathbf{b}_j\}$ satisfies $\mathbf{a}_i<\mathbf{b}_i$.

\begin{proposition}\label{prop:In Hilbert}
    Let $I$ be a homogeneous ideal of $S$ with reverse lexicographic monomial order. Then $\textrm{in}(I+(x_{i+1},\dots,x_n):x_i) = (\textrm{in}(I) + (x_{i+1},\dots,x_n):x_i)$. As a consequence, the modules
    \[
    (0:_{S/(I+(x_{i+1},\dots,x_n)} x_i) \textrm{ and } (0:_{S/(\textrm{in}(I)+(x_{i+1},\dots,x_n))} x_i)
    \]
    have the same Hilbert series.
\end{proposition}

\noindent There are analogous statements when some of the variables are taken to a power.

\begin{proposition}\label{prop:SS Pow AlmostReg}
    If $I$ is a strongly stable monomial ideal, then $x_n^{d_n},x_{n-1}^{d_{n-1}},\dots,x_1^{d_1}$ where $d_i\geq 1$ for each $i$ is an almost regular sequence for $S/I$.
\end{proposition}

\begin{proposition}\label{prop:Pow In Hilbert}
    Let $I$ be a homogeneous ideal of $S$ with reverse lexicographic monomial order. Then $\textrm{in}(I+(x_{i+1},\dots,x_n):x_i^{d_i}) = (\textrm{in}(I) + (x_{i+1},\dots,x_n):x_i^{d_i})$ for $d_i \geq 1$, and $\textrm{in}(I+(x_{i+1}^{d_{i+1}},\dots,x_n^{d_n}):x_i^{d_i}) \subset (\textrm{in}(I) + (x_{i+1}^{d_{i+1}},\dots,x_n^{d_n}):x_i^{d_i})$ for $d_i,\dots,d_n\geq 1$. As a consequence, the modules
    \[
    (0:_{S/(I+(x_{i+1},\dots,x_n)} x_i^{d_i}) \textrm{ and } (0:_{S/(\textrm{in}(I)+(x_{i+1},\dots,x_n)} x_i^{d_i})
    \]
    have the same Hilbert series, and the Hilbert series of $(0:_{S/(I+(x_{i+1}^{d_{i+1}},\dots,x_n^{d_n})} x_i^{d_i})$ is term-wise bounded above by the Hilbert series of $(0:_{S/(\textrm{in}(I)+(x_{i+1}^{d_{i+1}},\dots,x_n^{d_n}))} x_i^{d_i})$.
\end{proposition}

By combining Proposition \ref{prop:SS AlmostReg} and Proposition \ref{prop:In Hilbert}, one can deduce the following result, see \cite{aramova2000Almost} for more details using this proof method.

\begin{theorem}[\cite{bayer1999extremal}]\label{thm:GINextremalBettEq}
    The extremal Betti numbers and the extremal Betti positions of $S/I$ and $S/\textrm{GIN}(I)$ are the same.
\end{theorem}

\subsection{The multi-graded generic initial ideal}\label{sec:multi GIN}

Let $X_1,\dots,X_c$ be sets of variables where $X_j=\{x_{1,j},\dots,x_{n_j,j}\}$. Define the polynomial ring $R=\Bbbk [\bigcup_{j=1}^c X_j]$. We endow $R$ with the $\mathbb{N}^{c}$ multi-grading where the grading of the variable $x_{i,j}$ is the $c$-tuple with a 1 in the $j$'th position and zeros elsewhere.

Throughout, we always assume in any monomial order that 
\begin{align}\label{eq:Ordering}
    & x_{1,j}>\cdots>x_{n_j,j}
\end{align}
for every $1\leq j\leq c$.

Let $I$ be a multi-homogeneous ideal of $R$ with the reverse lexicographic monomial ordering induced by an ordering of variables satisfying (\ref{eq:Ordering}). Let $G= \textrm{GL}_\Bbbk (n_1)\times \cdots \times\textrm{GL}_\Bbbk (n_c)$. The \textit{$G$-generic initial ideal} of $I$ is the monomial ideal $\textrm{in}(uI)$ for a sufficiently generic matrix $u\in G$. The meaning of ``sufficiently generic'' is made precise by the following observation of Babson and Novak \cite{Babson2004Face}.

\begin{theorem}[\cite{Babson2004Face}]\label{thm:color zariski open}
    Let $I\subset R$ be a multi-homogeneous ideal. There is a Zariski open subset $U$ of $G$ and multi-homogeneous ideal $J=\textrm{$G$-GIN}(I)$ such that $\textrm{in}(uI)=J$ for all $u\in U$.
\end{theorem}

A useful fact is that the Zariski open condition in Theorems \ref{thm:GIN Open} and \ref{thm:color zariski open} imply that $J_j=\textrm{GIN}(I_j)$ for all $j$. 

Another necessary observation is that taking the $G$-generic initial ideal yields a \textit{strongly color-stable ideal}, which is an ideal satisfying property (\ref{eq:ColorStable}) below. It is analogous to the aforementioned strongly stable property.

\begin{theorem}[\cite{Babson2004Face}]\label{thm:GIN-ColorStronglyStable}
    Let $I\subset R$ be a multi-homogeneous ideal and assume $\Bbbk$ has characteristic $0$. Then $J=\textrm{$G$-GIN}(I)$ has the following property:
    \begin{equation}\label{eq:ColorStable}
    \textrm{if $mx_{i,j}\in J$ is a monomial, then $mx_{i',j}\in J$ for any $i'<i$.}
    \end{equation}
\end{theorem}

\subsection{Colored algebraic shifting}\label{sec:Alg Shifting}

In this section, we define colored algebraic shifting and present some of its useful properties. From here on, we assume that $\Bbbk$ has characteristic $0$. This ensures that the conclusion of Theorem \ref{thm:GIN-ColorStronglyStable} holds, which translates to the shifted property in Theorem \ref{thm:shifted props}. We note that the contents of this section recover the definitions and properties of the usual (uncolored) algebraic shifting when $c=1$. Let $\Delta$ be a simplicial complex whose vertex set is a disjoint union $V_1 \dot{\cup} \cdots \dot{\cup} V_c$. We may define the \textit{flag face numbers} of $\Delta$, denoted $f_{\mathbf{t}}$ for a tuple $\mathbf{t}\in \mathbb{N}^c$, as the number of faces with $\mathbf{t}_j$ vertices in $V_j$ for each $j$.

Let $I_\Delta$ be the Stanley-Reisner ideal of $\Delta$, and let $J=\textrm{$G$-GIN}(I_\Delta)$. In this special case (i.e. taking the $G$-generic initial ideal of a square-free monomial ideal), the minimal generators of $J$ satisfy a crucial property that we now describe.

\begin{definition}\label{def:A def}
    For a monomial $m$ of $R$, let $\textrm{max}_j(m)$ be the largest index $i$ such that $x_{i,j} | m$ (if there is no such $i$, we take $\textrm{max}_j(m)=0$), and let deg$_j(m)$ be the $j$'th entry of the multi-degree of $m$. Let $\mathscr{A}$ be the set of monomials of $R$ such that $\textrm{max}_j(m)+\textrm{deg}_j(m) \leq n_j+1$ for all $j$.
\end{definition}

\begin{theorem}[\cite{Murai2008Betti}]
    Let $I\subset R$ be a square-free monomial ideal. Then the minimal generators of $G$-GIN$(I)$ are contained in $\mathscr{A}$.
\end{theorem}

Now we define the \textit{colored square-free operator} $\tilde{\Phi}$ on monomial ideals of $R$ whose minimal generators are contained in $\mathscr{A}$. Let $x_{i_1,j}\cdots x_{i_k,j}$ be a monomial of degree $k$ where $i_1\leq i_2\leq \cdots \leq i_k$ and $i_k+k\leq n_j+1$. Then we define 
\[
\tilde{\Phi}(x_{i_1,j}\cdots x_{i_k,j})=x_{i_1,j}x_{i_2+1,j}\cdots x_{i_k+k-1,j}.
\]
The assumption that $i_k+k\leq n_j+1$ makes this well-defined. In general, given a monomial $m\in \mathscr{A}$, write $m=m_1\cdots m_c$ where $m_j$ consists of the variables in $X_j$ that divide $m$. Then we define 
\[
\tilde{\Phi}(m)=\tilde{\Phi}(m_1)\cdots \tilde{\Phi}(m_c).
\]
Finally, if $I$ is a monomial ideal whose minimal generators are contained in $\mathscr{A}$, then we define $\tilde{\Phi}(I)$ to be the square-free monomial ideal whose minimal generators are $\tilde{\Phi}(m)$ where $m$ is a minimal generator of $I$.
We can now define the \textit{colored algebraic shifting} of a simplicial complex.

\begin{definition}\label{def:algShift}
    Let $\Delta$ be a simplicial complex with vertex set $V=V_1\dot{\cup} \cdots \dot{\cup} V_c$. Then the simplicial complex $\Delta^{cs}$ whose Stanley-Reisner ideal is $\tilde{\Phi}(G\textrm{-GIN}(I_\Delta$)) is the \textit{colored algebraic shifting} of $\Delta$. 
\end{definition}

When $c=1$, Definition \ref{def:algShift} recovers the usual (uncolored) algebraic shifting, $\Delta^s$, of $\Delta$. Definition \ref{def:algShift} differs in notation from Definition 5.5 in \cite{Babson2004Face}. There, any monomial ordering such that $x_{1,j}>\cdots>x_{n_j,j}$ can be used, hence, the original notation includes the monomial ordering $>$. Here, we specify a  particular monomial ordering satisfying (\ref{eq:Ordering}) when needed, so we drop $>$ from the notation.

We now state the relevant useful properties of the colored algebraic shifting of a simplicial complex. We often associate the vertices of a simplicial complex $\Delta$ with the variables of $R$ for notational ease.

\begin{theorem}[\cite{Babson2004Face}]\label{thm:shifted props}

Let $\Delta$ be a simplicial complex whose vertex set is the disjoint union $V_1\dot{\cup} \cdots \dot{\cup} V_c$. Then the colored shifted complex $\Delta^{cs}$ satisfies the following properties:
\begin{enumerate}
    \item $\Delta^{cs}$ is color shifted: if $x_{i,j}$ is a vertex of a face $F$ in $\Delta^{cs}$, then $(F\setminus \{x_{i,j}\}) \cup \{x_{i',j}\}$ is a face of $\Delta^{cs}$ for every $i'>i$.
    \item $\Delta$ and $\Delta^{cs}$ have the same flag $f$-vector; $f_{\mathbf{t}}(\Delta)=f_{\mathbf{t}}(\Delta^{cs})$ for every $\mathbf{t}\in \mathbb{N}^c$.
\end{enumerate}
    
\end{theorem}
Additionally, the following result shows that the graded Betti numbers of a strongly color-stable monomial ideal $I$ whose minimal generators are contained in $\mathscr{A}$ coincide with the graded Betti numbers of $\tilde{\Phi}(I)$. The original uncolored version of this result is proven in \cite{aramova2000shifting}.

In particular, when computing the graded Betti numbers of $R/I_{\Delta^{cs}}$, one can work with $R/G\textrm{-GIN}(I_\Delta)$ instead.

\begin{theorem}[\cite{Murai2008Betti}]\label{thm:Murai shifted betti}
    Let $I\subset R$ be a strongly color-stable monomial ideal whose minimal generators are contained in $\mathscr{A}$. Then $\beta_{i,j}(I)=\beta_{i,j}(\tilde{\Phi}(I))$ for all $i$ and $j$.
\end{theorem}

\section{Proofs of the main results}
\subsection{Proof of Theorem \ref{thm:Alg Main}}

Recall that $R=\Bbbk [\bigcup_{j=1}^c X_j]$ is a multi-graded polynomial ring, and for a multi-homogeneous ideal $I\subset R$ and an index $1\leq j\leq c$, we set $I_j=I\cap \Bbbk[X_j]$.

We first prove a couple of supporting lemmas. Lemma \ref{lem:Koszul Build} shows how to build a cycle in the Koszul complex for $R/I$ given cycles in the Koszul complexes for the $R_j/I_j$'s. Let $\ell_1^{(j)},\dots,\ell_{n_j}^{(j)}$ be a regular sequence for $R_j$ for all $1\leq j\leq c$. Define $W_j=\{(1,j),\dots,(n_j,j) \}$ and $W=\bigcup_{j=1}^c W_j$. We denote a standard basis element in the Koszul complex $K_\bullet (\ell_1^{(j)},\dots,\ell_{n_j}^{(j)};R_j/I_j)$ by $\mathbf{e}_A$ where $A\subset W_j$. Similarly, we denote a standard basis element in the Koszul complex $K_\bullet (\ell_1^{(1)},\dots,\ell_{n_1}^{(1)},\dots,\ell_1^{(c)},\dots,\ell_{n_c}^{(c)};R/I)$ by $\mathbf{e}_A$ where $A\subset W$. Note also that for elements $\alpha_t\in R_{j_t}/I_{j_t}$, $1\leq t\leq k$, there is a well-defined multiplication $\alpha_1\cdots \alpha_k \in R/I$. 

\begin{lemma}\label{lem:Koszul Build}
    Let $z_p=\sum_t \alpha_t^{(p)}\mathbf{e}_{A_t^{p}}$ be a cycle in the Koszul complex $K_\bullet (\ell_1^{(j_p)},\dots,\ell_{n_{j_p}}^{(j_p)};R_{j_p}/I_{j_p})$ for $1\leq p\leq k$. Then the element
    \[
    z=\sum_{(t_1,\dots,t_k)}  \alpha_{t_1}^{(1)}\cdots \alpha_{t_k}^{(k)} \mathbf{e}_{A_{t_1}^{1}\cup \cdots \cup A_{t_k}^{k}}
    \]
    is a cycle in the Koszul complex $K_\bullet (\ell_1^{(1)},\dots,\ell_{n_1}^{(1)},\dots,\ell_1^{(c)},\dots,\ell_{n_c}^{(c)};R/I)$.
\end{lemma}
\begin{proof}
    The differential of $z$ is given by
    \[
    \partial(z) = \pm\partial(z_1)\left(\sum_{(t_2,\dots,t_k)}  \alpha_{t_2}^{(2)}\cdots \alpha_{t_k}^{(k)} \mathbf{e}_{A_{t_2}^{2}\cup \cdots \cup A_{t_k}^{k}}\right) \pm \cdots \pm \partial(z_k)\left(\sum_{(t_1,\dots,t_{k-1})}  \alpha_{t_1}^{(1)}\cdots \alpha_{t_{k-1}}^{(k-1)} \mathbf{e}_{A_{t_1}^{1}\cup \cdots \cup A_{t_{k-1}}^{k-1}}\right).
    \]
    Since $\partial(z_1)=\cdots =\partial(z_k)=0$, we have that $\partial(z)=0$ and hence $z$ is a cycle.
\end{proof}

The following lemma is a consequence of the reverse lexicographic monomial order. 

\begin{lemma}\label{lem:ordering lem}
    Let $f \in R$ be multi-homogeneous, and suppose that we take the reverse lexicographic monomial order where the variables $\{x_{1,j},\dots,x_{i_j,j} \mid 1\leq j \leq c\}$ are greater than any variable in $\{x_{i_j+1,j},\dots,x_{n_j,j} \mid 1\leq j \leq c\}$. If $\textrm{in}(f)\in \Bbbk[x_{1,j},\dots,x_{i_j,j} \mid 1\leq j \leq c]$, then for every $h\in (x_{i_j+1,j},\dots,x_{n_j,j} \mid 1\leq j \leq c)$ with the same multi-degree as $f$, $\textrm{in}(f+h)=\textrm{in}(f)$.
\end{lemma}
\begin{proof}
    This follows directly from the definition of the reverse lexicographic ordering.
\end{proof}

\AlgMain*
\begin{proof}
    First note, it follows from the Zariski open condition in Theorems \ref{thm:GIN Open} and \ref{thm:color zariski open}, that $J_j=\textrm{GIN}(I_j)$ for each $j$. After replacing $I$ with $uI$ for $G$-generic $u$, we may assume simply that $J=\textrm{in}(I)$ and $J_j=\textrm{in}(I_j)$. 

    We assume that $j_1=1,\dots,j_k = k$ for simplicity. We also set $d_{i_j+1,j}=\cdots =d_{n_j,j}=1$ for all $1\leq j\leq k$. By Proposition \ref{prop:SS Pow AlmostReg}, $x_{n_{j},j}^{d_{n_{j},j}},\dots,x_{1,j}^{d_{1,j}}$ is an almost regular sequence of $R_{j}/J_{j}$ and hence of $R_{j}/I_{j}$ as well by Proposition \ref{prop:Pow In Hilbert}. For the rest of the proof, we set $\ell_1^{(j)} = x_{n_{j},j}^{d_{n_{j},j}},\dots,\ell_{n_{j}}^{(j)} = x_{1,j}^{d_{1,j}}$. By Proposition \ref{prop:Pow In Hilbert}, there exists an element $f_j \in \left(I_{j} + (\ell_1^{(j)},\dots,\ell_{n_{j}-i_j}^{(j)}):_{R_{j}} \ell_{n_{j}-i_j+1}^{(j)}\right)_{q_j}$ such that $\textrm{in}(f_j)=a_j$ for all $1\leq j\leq k$. Also, the fact that $a_1\cdots a_k\notin J+\left(\ell_1^{(j)},\dots,\ell_{n_{j}-i_j}^{(j)} \mid 1\leq j\leq k\right)$ implies that $a_j \notin J_{j}+\left(\ell_1^{(j)},\dots,\ell_{n_{j}-i_j}^{(j)} \right)$ for all $j$. Therefore, $f_j \notin I_{j}+\left(\ell_1^{(j)},\dots,\ell_{n_{j}-i_j}^{(j)} \right)$ as well. Proposition \ref{prop:Kosz Hom rep} can then be applied with $M=R_{j}/I_{j}$: there exists a nonzero element of $H_{i_j}(\ell_1^{(j)},\dots,\ell_{n_{j}}^{(j)};R_{j}/I_{j})$ represented by an element of the form $(f_j\mathbf{e}_{A_1^j}+\cdots)\in K_{i_j}\left(\ell_1^{(j)},\dots,\ell_{n_{j}}^{(j)};R_{j}/I_{j}\right)$ where $A_1^j=\{(n_{j}-i_j+1,j),\dots,(n_{j},j) \}$. For now, write this element as $\sum_t \alpha_t^{(j)}\mathbf{e}_{A_t^j}$. By Lemma \ref{lem:Koszul Build}, we have the cycle 
    \[
    z=\sum_{(t_1,\dots,t_k)}  \alpha_{t_1}^{(1)}\cdots \alpha_{t_k}^{(k)} \mathbf{e}_{A_{t_1}^{1}\cup \cdots \cup A_{t_k}^{k}} \in K_{\sum_ji_j}\left(\ell_1^{(1)},\dots,\ell_{n_{1}}^{(1)},\dots,\ell_1^{(k)},\dots,\ell_{n_{k}}^{(k)}; R/I \right).
    \]
    To see that it is a nontrivial cycle, note that $f_1\cdots f_k \mathbf{e}_{A_1^1\cup \cdots \cup A_1^k}$ is a summand of $z$, and observe that $$\textrm{in}(f_1\cdots f_k) = a_1\cdots a_k.$$ 
    Therefore, if $z$ is a boundary, then $f_1\cdots f_k\in I+\left(\ell_1^{(j)},\dots,\ell_{n_{j}-i_j}^{(j)} \mid 1\leq j\leq k\right)$. In other words, there exists some $h\in \left(x_{i_j+1,j},\dots,x_{n_{j},j} \mid 1\leq j\leq k\right)$ such that $f_1\cdots f_k+h\in I$. However, by Lemma \ref{lem:ordering lem}, $\textrm{in}(f_1\cdots f_k+h)=a_1\cdots a_k$, and hence $a_1\cdots a_k \in \textrm{in}(I)=J$, a contradiction. Thus, $z$ represents a nontrivial cycle.

    Observe that the (total) degree of $z$ is $q_1+\cdots+q_k+\sum_{j=1}^k \sum_{p=1}^{i_j} d_{p,j}$ and that the multi-grading of $z$ is a $c$-tuple that is only nonzero in positions $1,\dots,k$. The top (total) degree of $R/(x_{1,j}^{d_{1,j}},\dots,x_{n_{j},j}^{d_{n_{j},j}} \mid 1\leq j\leq k)$ in such multi-gradings is $\sum_{j=1}^k \sum_{p=1}^{n_{j}} (d_{p,j}-1)$. Now, by a similar argument as in Proposition \ref{prop:Reg Lower} and its proof, the regularity of $R/I$ must be at least 
    \[
    q_1+\cdots+q_k+\sum_{j=1}^k \sum_{p=1}^{i_j} d_{p,j} -\sum_{j=1}^ki_j - \sum_{j=1}^k \sum_{p=1}^{n_{j}} (d_{p,j}-1) = q_1+\cdots+q_k-\sum_{j=1}^k\sum_{p=i_j+1}^{n_{j}} (d_{p,j}-1)=q_1+\cdots+q_k. 
    \]
\end{proof}

\subsection{Proof of Theorem \ref{thm:comb main}}\label{sec:Comb main sec}

For a $c$-tuple $\mathbf{t}\in \mathbb{N}^c$, let $X_{\mathbf{t},j}=\{x_{n_j-r_j+\min\{\mathbf{t}_j,d+1\},j},\dots,x_{n_j,j}\}$. For a subset $F$, define 
\[
g_{\mathbf{t},j}(F)=\begin{cases}
|F\cap (X_j\setminus{X_{\mathbf{t},j}})| & \textrm{if } |F\cap (X_j\setminus{X_{\mathbf{t},j}})|\geq \min\{\mathbf{t}_j,d+1\},\\
 0 &\textrm {otherwise.}
\end{cases}
\]

As a supporting lemma, we first prove the uncolored version of Theorem \ref{thm:comb main} using algebraic shifting. In fact, Theorem \ref{thm:d-leray-upbound} is a direct consequence of Lemma \ref{lem: 1- color comb main}.

\begin{lemma}\label{lem: 1- color comb main}
    Let $\Delta$ be a $d$-Leray simplicial complex with $n$ vertices, and let $\Delta^s$ be its algebraic shifting. If the largest face of $\Delta$ has size $r$, then every face of $\Delta^s$ contains at most $d$ vertices outside of $\{n-(r-d)+1,\dots,n\}$.
\end{lemma}
\begin{proof}
    First, we observe that $\Delta^s$ is $d$-Leray. This follows from Theorem \ref{thm:BayerStillman} together with \ref{thm:Murai shifted betti}.
    
    Assume some face of $\Delta^s$ has $d+1$ vertices outside of $\{n-(r-d)+1,\dots,n\}$. Then, by the shifted property of $\Delta^s$, $\{n-r,\dots,n-(r-d) \}$ is a face of $\Delta^s$. Again by the shifted property, every $(d+1)$-set of $\{n-r,\dots,n-(r-d)+1 \}$ is a face of $\Delta^s$. The $d$-Leray property then implies that $\{n-r,\dots,n-(r-d)+1 \}$ is a face of $\Delta^s$. Continuing in this manner, we get that $\{n-r,\dots,n \}$ is a face of $\Delta^s$. This is a contradiction since the largest face of $\Delta^s$ has size  $r$.
    \end{proof}

\CombMain*
\begin{proof}
    We first observe that it follows from the Zariski open condition in Theorems \ref{thm:GIN Open} and \ref{thm:color zariski open} that $\Delta^{cs}[V_j]=(\Delta[V_j])^s$ for all $j$.
    
    We assume for simplicity that $j=1,\dots,k$ is the set of indices such that $g_{\mathbf{t},j}(F) \neq 0$, and we assume for contradiction that $F$ is a face of $\Delta^{cs}$. By Lemma \ref{lem: 1- color comb main} and the fact that $\Delta^{cs}[V_j]=(\Delta[V_j])^s$, we must have that $x_{n_j-r_j,j}^{d+1}\in J=G\textrm{-GIN}(I_\Delta)$  and hence $\mathbf{t}_j\leq d$ for all $1\leq j\leq k$.  Since $\Delta^{cs}$ is color shifted, $\Delta^{cs}$ has the face $\{x_{n_1-r_1,1},\dots,x_{n_1-r_1+\mathbf{t}_1-1,1} \}\cup \cdots \cup \{ x_{n_k-r_k,k},\dots,x_{n_k-r_k+\mathbf{t}_k-1,k} \}$. In other words, $x_{n_1-r_1,1}^{\mathbf{t}_1}\cdots x_{n_k-r_k,k}^{\mathbf{t}_k}$ is not contained in $J$. 

    We will apply Theorem \ref{thm:Alg Main} to complete the proof. We take $d_{n_j,j}=\cdots = d_{n_j-r_j+1,j}=1$ and $d_{n_j-r_j,j}=\cdots =d_{1,j}=(d+1)-\mathbf{t}_j$ for $1\leq j\leq k$.
    We have that 
    \begin{align*}
    &x^{\mathbf{t}_j}_{n_j-r_j,j} \in (J_j+(x_{n_j,j},\dots,x_{n_j-r_j+1,j}):_{R_j} x^{(d+1)-\mathbf{t}_j}_{n_j-r_j,j}),\\
    &x_{n_1-r_1,1}^{\mathbf{t}_1}\cdots x_{n_k-r_k,k}^{\mathbf{t}_k}\notin J + (x_{n_j,j},\dots,x_{n_j-r_j+1,j}\mid 1\leq j\leq k), \textrm{ and}\\
    & \mathbf{t}_j + ((d+1) - \mathbf{t}_j) = d+1 > \textrm{reg}(R_j/(I_\Delta)_j).
    \end{align*}
Theorem \ref{thm:Alg Main} then implies $\textrm{reg}(R/I_\Delta) \geq \sum_{j=1}^k \mathbf{t}_j = g_{\mathbf{t},j}(F) > d$. This is a contradiction and completes the proof.
\end{proof}

A direct application of Theorem \ref{thm:comb main} gives a proof of Conjecture \ref{conj:Upper Main} and hence of Conjecture \ref{conj:col frac helly}.

\ConjecturesHold*
\begin{proof}
    By Theorem \ref{thm:comb main}, no face in ${V \choose \mathbf{1}}$ of $\Delta^{cs}$ is contained in 
    \[
    \{x_{1,1},\dots,x_{n_1-r_1,1} \}\cup \cdots \cup \{x_{1,d+1},\dots,x_{n_{d+1}-r_{d+1},d+1}\}.
    \] 
    Therefore, the number of faces in ${V \choose \mathbf{1}}$ of $\Delta^{cs}$ is at most $n_1\cdots n_{d+1} - (n_1-r_1)\cdots (n_{d+1}-r_{d+1})$.
\end{proof}
\noindent Finally, we show via Proposition \ref{prop:Upper Bound Tight} below that the more general upper bound in Corollary \ref{cor:gen Upper Bound} is tight. For this, we briefly recall the notion of $d$-collapsibility, see \cite{martin2013intersection} for more information. Notably, a $d$-collapsible simplicial complex is $d$-Leray.

A face $\sigma$ of a simplicial complex $\Delta$ such that  $\sigma$ is contained in a unique maximal face of $\Delta$ and $|\sigma| \leq d$ is called $d$-collapsible. The simplicial complex $\Delta' = \Delta \setminus \{\tau\in \Delta \mid \tau \supseteq \sigma \}$ is an elementary $d$-collapse of $\Delta$. Finally, we say that $\Delta$ is $d$-collapsible if there is a sequence of $d$-collapses that reduce $\Delta$ to the empty complex.

\begin{proposition}\label{prop:Upper Bound Tight}
    Let $V=V_1\dot{\cup} \cdots \dot{\cup} V_c$, and let $\mathbf{n}=(n_1,\dots,n_c)$, $\mathbf{r}=(r_1,\dots,r_c)$ be tuples such that $r_i\leq n_i$ for all $i$. Then the simplicial complex
    \[
    \Delta = \left\{F\subseteq V \mid F \textrm{ contains no set in } \mathscr{F}_{\mathbf{k}}(\mathbf{n},d,\mathbf{r}) \textrm{ for any }\mathbf{k}\in \mathbb{N}^c  \right\}
    \]
    is a $d$-collapsible simplicial complex.
\end{proposition}
\begin{proof}
    Let $F\in {V \choose \mathbf{t}}$ be a face of size $d$ such that $\sum g_{\mathbf{t},j}(F) = d$. Then it is easy to see that $F \cup (\bigcup_{j=1}^c X_{\mathbf{t}+\mathbf{1},j})$ is the unique maximal face containing $F$.  Thus, we may perform the elementary $d$-collapse at the face $F$. We may do so for all such faces satisfying $\sum g_{\mathbf{t},j}(F) = d$; this is because if $F\in {V \choose \mathbf{t}}$ and $F'\in {V \choose \mathbf{t}'}$ are distinct faces of size $d$ such that $\sum g_{\mathbf{t},j}(F) = g_{\mathbf{t}',j}(F') = d$, then there is no face that contains both of them. After these collapses, we consider faces $F\in {V \choose \mathbf{t}}$ of size $d-1$ such that $\sum g_{\mathbf{t},j}(F) = d-1$. Similarly, the unique maximal face containing $F$ is $F \cup (\bigcup_{j=1}^c X_{\mathbf{t}+\mathbf{1},j})$ and we collapse all such faces. Continuing in this way, we eventually collapse down to the simplex $\bigcup_{j=1}^c X_{\mathbf{1},j}$, which is of course $d$-collapsible. Therefore, $\Delta$ is $d$-collapsible, and this completes the proof.
\end{proof}

\subsection{Proof of Theorem \ref{thm:Betti lower bound}}\label{sec:ME-betti}

We first prove Theorem \ref{thm:Koszul containment} below from which Theorem \ref{thm:Betti lower bound} can be derived. The motivation comes from the operation in Lemma \ref{lem:Koszul Build} together with Proposition \ref{prop:Kosz Hom iso}.

\begin{theorem}\label{thm:Koszul containment}
    Let $1\leq j_1 <\cdots < j_k\leq c$ be integers, and take the reverse lexicographic monomial order where the variables $\{x_{1,{j_t}},\dots,x_{i_{j_t},{j_t}} \mid 1\leq t \leq k\}$ are greater than any variable in $\{x_{i_{j_t}+1,{j_t}},\dots,x_{n_{j_t},{j_t}} \mid 1\leq t \leq k\}$.
    Let $I\subset R$ be a multi-homogeneous ideal and $J=G\textrm{-GIN}(I)$. Suppose $$({i_{j_1},i_{j_1}+s_{j_1}}),\dots,({i_{j_k},i_{j_k}+s_{j_k}})$$ are extremal Betti positions of $R_{j_1}/J_{j_1},\dots,R_{j_k}/J_{j_k}$, respectively. Consider the vector space map
    \[
    \Psi:(0:_{M_{j_1}\langle n_{j_1}-i_{j_1} \rangle}x_{i_{j_1},j_1})_{s_{j_1}}\otimes \cdots \otimes (0:_{M_{j_k}\langle n_{j_k}-i_{j_k} \rangle }x_{i_{j_k},j_k})_{s_{j_k}} \longrightarrow (0:_{M} (x_{i_{j_1},j_1},\dots,x_{i_{j_k},j_k}))_{s_{j_1}+\cdots +s_{j_k}}
    \]
    given by $g_1 \otimes \cdots \otimes g_k \mapsto g_1\cdots g_k$, where $M_{j_t}\langle n_{j_t}-i_{j_t} \rangle=R_{j_t}/(J_{j_t} + (x_{n_{j_t},j_t},\dots,x_{i_{j_t}+1,j_t}))$ for $1\leq t\leq k$ and $M=R/(J+(x_{n_{j_t},j_t},\dots,x_{i_{j_t}+1,j_t}\mid 1\leq t\leq k))$. Then the subspace of 
    $$
    H_{\sum_t i_{j_t}}(x_{n_1,1},\dots,x_{1,1},\dots,x_{n_c,c},\dots,x_{1,c};R/I)
    $$
    with multi-grading $\mathbf{t}$, where $\mathbf{t}_p=i_{j_t}+s_{j_t}$ if $p=j_t$ and  $\mathbf{t}_p=0$ otherwise, is at least dim$_\Bbbk\textrm{Im}(\Psi)$-dimensional.
\end{theorem}
\begin{proof}
    As usual, after replacing $I$ with $uI$ for $G$-generic $u$, we may assume that $J=\textrm{in}(I)$. Let $a_1,\dots,a_d$ be monomials representing a basis of $\textrm{Im}(\Psi).$ Then for each $a_t$, we have representative monomials $a_{1,t}\in (0:_{M_{j_1}\langle n_{j_1}-i_{j_1} \rangle}x_{i_{j_1},j_1})_{s_{j_1}},\dots,a_{k,t}\in (0:_{M_{j_k}\langle n_{j_k}-i_{j_k}\rangle }x_{i_{j_k},j_k})_{s_{j_k}}$ such that $a_{1,t}\cdots a_{k,t}=a_t$. By Proposition \ref{prop:In Hilbert}, there exists $f_{1,t}\in (0:_{N_{j_1}\langle n_{j_1}-i_{j_1} \rangle}x_{i_{j_1},j_1})_{s_{j_1}},\dots,f_{k,t}\in (0:_{N_{j_k}\langle n_{j_k}-i_{j_k}\rangle }x_{i_{j_k},j_k})_{s_{j_k}}$, where $N_{j_t}\langle n_{j_t}-i_{j_t} \rangle=R_{j_t}/(I_{j_t} + (x_{n_{j_t},j_t},\dots,x_{i_{j_t}+1,j_t}))$, such that $\textrm{in}(f_{1,t})=a_{1,t},\dots,\textrm{in}(f_{k,t})=a_{k,t}$.

    By Proposition \ref{prop:Kosz Hom iso}, there exist cycles of the form $(f_{1,t}\mathbf{e}_{A_1^{1}}+\cdots),\dots,(f_{k,t}\mathbf{e}_{A_1^{k}}+\cdots)$ of the respective Koszul complexes, where $A^t_1 = \{(n_{j_t}-i_{j_t}+1,j_t),\dots,(n_{j_t},j_t) \}$. By Lemma \ref{lem:Koszul Build}, there is a cycle of the form
    \[
    z_t=(f_{1,t}\cdots f_{k,t}\mathbf{e}_{A_1^1\cup \cdots \cup A_1^k}+\cdots)\in K_{\sum_t i_{j_t}}(x_{n_1,1},\dots,x_{1,1},\dots,x_{n_c,c},\dots,x_{1,c};R/I).
    \]
    We show that $z_1,\dots,z_d$ represent linearly independent elements in Koszul homology $$H_{\sum_t i_{j_t}}(x_{n_1,1},\dots,x_{1,1},\dots,x_{n_c,c},\dots,x_{1,c};R/I).$$ If $\sum_{t}\alpha_t z_t$ is a boundary, then $f=\sum_{t}\alpha_t f_{1,t}\cdots f_{k,t}$ lies in $I+(x_{n_{j_t},j_t},\dots,x_{i_{j_t}+1,j_t}\mid 1\leq t\leq k)$, in other words, there exists $h\in (x_{n_{j_t},j_t},\dots,x_{i_{j_t}+1,j_t}\mid 1\leq t\leq k)$ such that $f+h\in I$. Lemma \ref{lem:ordering lem} implies that $\textrm{in}(f+h) = \textrm{in}(f) \in \{a_1,\dots,a_d\}$, which is a contradiction. Therefore, the dimension of $$H_{\sum_t i_{j_t}}(x_{n_1,1},\dots,x_{1,1},\dots,x_{n_c,c},\dots,x_{1,c};R/I)_{\mathbf{t}},$$ where $\mathbf{t}$ is as in the theorem statement, is at least $d$.
    
\end{proof}

Recall for a tuple $\mathbf{t}\in \mathbb{N}^c$, we impose further the following condition on the ordering of the variables: the variables in $\{x_{1,j},\dots,x_{n_j-\mathbf{t}_j,j} \}$ are greater than each variable in $\{x_{n_j-\mathbf{t}_j+1,j},\dots,x_{n_j,j} \}$ for each $j$. We fix a monomial order $>_\mathbf{t}$ that satisfies this ordering condition on the variables, and we define $\Delta_{\mathbf{t}}^{cs}$ to be the colored algebraic shifting with respect to this order.

\bettiLowBound*
\begin{proof}
    Fix a tuple $\mathbf{t}\in \mathbb{N}^c$ such that $\sum_{j=1}^c\mathbf{t}_j=i$ and a face $F\in \Delta_{\mathbf{t}}^{cs}$ such that $|F\cap V_j| = \mathbf{t}_j \textrm{ and } (F\cap V_j)\cup \{x_{n_j,j}\}\notin \Delta_{\mathbf{t}}^{cs}$ for all $j$. Take the reverse lexicographic monomial ordering where the variables in $\{x_{1,j},\dots,x_{n_j-\mathbf{t}_j,j} \}$ are greater than each variable in $\{x_{n_j-\mathbf{t}_j+1,j},\dots,x_{n_j,j} \}$ for each $j$. We apply Theorem \ref{thm:Koszul containment} with $I=I_\Delta$ and $J_{\mathbf{t}}=G\textrm{-GIN}(I_\Delta)$ with respect to the ordering $>_{\mathbf{t}}$. It follows from the Zariski open condition in Theorem \ref{thm:color zariski open} that that there exists $u$ such that $\textrm{in}_{>_\mathbf{t}}(uI) = J_{\mathbf{t}}$ for each such tuple $\mathbf{t}$. Hence, we may replace $I$ with $uI$, and it follows that $J_{\mathbf{t}} = \textrm{in}_{>_{\mathbf{t}}}(I)$ for all $\mathbf{t}$. 
    
    For each $j$, let $a_j$ be the monomial of degree $\mathbf{t}_j$ such that the support of $\tilde{\Phi}(a_j)$ is $F\cap V_j$. Since $(F\cap V_j)\cup \{x_{n_j,j}\}\notin \Delta_{\mathbf{t}}^{cs}$, we have that $a_jx_{n_j-\mathbf{t}_j,j}\in J_{\mathbf{t}}$. In other words, $a_j\in (0:_{M_{j}\langle \mathbf{t}_j \rangle}x_{n_j-\mathbf{t}_j,j})_{\mathbf{t}_{j}}$. It follows from Hochster's Formula that $(n_j-\mathbf{t}_j,n_j)$ is an extremal Betti position of $R_j/(I_\Delta)_j$ and hence of $R_j/(J_{\mathbf{t}})_j$ by Theorem \ref{thm:GINextremalBettEq}. Additionally, $a_1\dots a_c$ represents a nonzero element of $(0:_{M} (x_{n_1-\mathbf{t}_1,1},\dots,x_{n_c-\mathbf{t}_c,c}))_{\mathbf{t}_{1}+\cdots +\mathbf{t}_{c}}$ since $F\in \Delta_{\mathbf{t}}^{cs}$. Therefore, the image of the map
    \[
    (0:_{M_{1}\langle \mathbf{t}_1 \rangle}x_{n_1-\mathbf{t}_1,1})_{\mathbf{t}_{1}}\otimes \cdots \otimes (0:_{M_{c}\langle \mathbf{t}_c \rangle}x_{n_c-\mathbf{t}_c,c})_{\mathbf{t}_{c}} \longrightarrow (0:_{M} (x_{n_1-\mathbf{t}_1,1},\dots,x_{n_c-\mathbf{t}_c,c}))_{\mathbf{t}_{1}+\cdots +\mathbf{t}_{c}}
    \]
    has dimension at least $\left|\{ F\in \Delta_{\mathbf{t}}^{cs}\cap {V \choose \mathbf{t}} \mid (F\cap V_j)\cup \{x_{n_j,j}\}\notin \Delta_{\mathbf{t}}^{cs} \textrm{ for all } j \} \right|$ and lies in the multi-grading $\mathbf{t}$. As in the proof of Theorem \ref{thm:Koszul containment}, $H_{n-i}(x_{n_1,1},\dots,x_{1,1},\dots,x_{n_c,c},\dots,x_{1,c};R/I)_{\mathbf{n}}$ contains a subspace of dimension at least $\left|\{ F\in \Delta_{\mathbf{t}}^{cs}\cap {V \choose \mathbf{t}} \mid (F\cap V_j)\cup \{x_{n_j,j}\}\notin \Delta_{\mathbf{t}}^{cs} \textrm{ for all } j \} \right|$ represented by elements of the form $(f_{\mathbf{t}}\mathbf{e}_{A_{\mathbf{t}}}+\cdots)$, where $A_{\mathbf{t}}=\bigcup_{j=1}^c \{ (\mathbf{t}_j,j),\dots,(n_j,j)\}$ and $f_{\mathbf{t}}\notin I + (x_{n_j,j},\dots,x_{n_j-\mathbf{t}_j+1,j}\mid 1\leq j\leq c)$. We claim that the direct sum of these subspaces is at least $\sum_{\mathbf{t}}\left|\{ F\in \Delta_{\mathbf{t}}^{cs}\cap {V \choose \mathbf{t}} \mid (F\cap V_j)\cup \{x_{n_j,j}\}\notin \Delta_{\mathbf{t}}^{cs} \textrm{ for all } j \} \right|$-dimensional. Indeed, if an element $\sum_{\mathbf{t}}\alpha_{\mathbf{t}}(f_{\mathbf{t}}\mathbf{e}_{A_{\mathbf{t}}}+\cdots)$ is a boundary, then this directly implies that $f_{\mathbf{t}}\in I + (x_{n_j,j},\dots,x_{n_j-\mathbf{t}_j+1,j}\mid 1\leq j\leq c)$, which is of course a contradiction. Since $H_{n-i}(x_{n_1,1},\dots,x_{1,1},\dots,x_{n_c,c},\dots,x_{1,c};R/I)_{\mathbf{n}}$ has dimension $\beta_{i-1}(\Delta)$ by Hochster's Formula, this completes the proof.
\end{proof}

\section{Acknowledgments}
This work was supported by the Institute for Basic Science (IBS-R029-C1). It was also partially supported by the National Science Foundation (NSF) under award no. 2402145.

We are grateful to Jason McCullough for helpful discussions in the early stages of this project. We are also thankful to Andreas Holmsen and Nikola Sadovek for comments and suggestions on an earlier version of this paper.

\bibliographystyle{plain}
\bibliography{bibliography.bib}

\end{document}